\documentclass[a4paper]{amsart}%

\usepackage[inline]{fixme}
\fxsetup{inline=false}

\usepackage{hyperref}
\usepackage[tikz]{tilmansdef}
\usetikzlibrary{calc,cd}

\usepackage{leftidx}

\newcommand{\st}{\mathcal A}
\def\SO{\mathrm{SO}}
\def\Cl{Cl}
\DeclareMathOperator{\Mat}{Mat}
\def\kO{k\mathrm{O}}
\def\tmf{\mathrm{tmf}}
\renewcommand{\O}{\mathbf O}
\def\GG{\mathrm{G}}
\def\DW{\mathrm{DW}}

\def\dbl#1(#2){\Phi^{#2}#1}

\def\geq{\geqslant}
\def\leq{\leqslant}

\title[The realizability of some finite-length
Steenrod modules]
{The realizability of some finite-length modules
over the Steenrod algebra by spaces}

\author{Andrew Baker}
\address{
School of Mathematics \& Statistics,
University of Glasgow, Glasgow G12~8QQ, Scotland.}
\email{a.baker@maths.gla.ac.uk}
\urladdr{http://www.maths.gla.ac.uk/~ajb}
\author{Tilman Bauer}
\address{
Department of Mathematics, School of Engineering Sciences,
KTH Royal Institute of Technology, Lindstedtsv\"agen 25,
10044 Stockholm, Sweden.
}
\email{tilmanb@kth.se}
\urladdr{https://people.kth.se/~tilmanb}


\date\today

\tikzset{%
gen/.style={circle,fill,inner sep=0pt,minimum size=1.5mm}
}
\tikzset{%
sq1/.style={}
}
\tikzset{%
sq2r/.style = {bend right=45}
}
\tikzset{%
sq2l/.style = {bend left=45}
}
\tikzset{%
sq4r/.style = {to path={-- +(.6,.3)  -- +(.6,3.7) -- (\tikztotarget)}}
}
\tikzset{%
sq4l/.style = {to path={-- +(-.6,.3) -- +(-.6,3.7) -- (\tikztotarget)}}
}
\tikzset{%
sq8r/.style = {to path={-- +(.6,.3)  -- +(.6,7.7) -- (\tikztotarget)}}
}
\tikzset{%
sq8l/.style = {to path={-- +(-.6,.3) -- +(-.6,7.7) -- (\tikztotarget)}}
}
\begin{document}
\begin{abstract}
The Joker is an important finite cyclic module over the mod-$2$
Steenrod algebra $\st$. We show that the Joker, its first two
iterated Steenrod doubles, and their linear duals are realizable
by spaces of as low a dimension as the instability condition of
modules over the Steenrod algebra permits. This continues and
concludes prior work by the first author and yields a complete
characterization of which versions of Jokers are realizable by
spaces or spectra and which are not. The constructions involve
sporadic phenomena in homotopy theory ($2$-compact groups,
topological modular forms) and may be of independent interest.
\end{abstract}
\maketitle

\section{Introduction}

Let $\st$ be the mod-$2$ Steenrod algebra, minimally generated
by the Steenrod squares $\Sq^{2^s}$ ($s\geq0$), and
$\st(n)\subseteq\st$, for $n\geq0$, the finite sub-Hopf algebra
generated as an algebra by $\{\Sq^{2^s} \mid s\leq n\}$.

A (left) $\st$-module~$M$ is \emph{stably realizable}
if there exists a spectrum $X$ such that as $\st$-modules,
\[
H^*(X) \underset{\text{def}}{=} H^*(X;\F_2) \cong M.
\]
For finite $\st$-modules, this is equivalent to the existence
of a space $Z$ such that $\widetilde{H}^*(Z) \cong \Sigma^s M$
for some~$s$. This number $s$ is bounded from below by the
\emph{unstable degree} $\sigma(M)$ of $M$, i.e. the minimal
number~$t$ such that $\Sigma^tM$ satisfies the instability
condition for modules over $\st$. We say that $M$ is
\emph{optimally realizable} if there exists a space $Z$
such that $\widetilde{H}^*(Z) \cong \Sigma^{\sigma(M)}M$.

We consider two constructions of new Steenrod modules from
old. Firstly, for a left $\st$-module $M$, the linear dual
$M^\vee=\Hom(M,\F_2)$ becomes a left $\st$-module using
the antipode of $\st$. Secondly, the \emph{iterated
double} $\dbl M(i)$ is the module which satisfies
\begin{align*}
\dbl M(i)^n =
\begin{cases}
M^{n/2^i} & \text{if $2^i \mid n$}, \\
0 & \text{otherwise},
\end{cases} \\
\intertext{and for $x \in \dbl M(i)^n$,}
\Sq^{2^k}x =
\begin{cases}
0 & \text{if $k<i$}, \\
\Sq^{2^{k-i}}x & \text{if $k \geq i$}.
\end{cases}
\end{align*}
We also set $\dbl M(0)=M$.

Let $J$ be the quotient of $\st$ by the left
ideal generated by $\Sq^3$ and $\Sq^{i}\st$
for $i \geq 4$.

The main result of this paper is the following.
\begin{thm}\label{thm:main}
The modules $\dbl J(i)$ and $\dbl J(i)^\vee$
are optimally realizable for $i \leq 2$ and
not stably realizable for $i >2$.
\end{thm}

The module $J$ in this theorem is known as
the \emph{Joker}, although more colloquially
than in actual written articles.

In \cite{baker:joker}, the first author showed
all cases of Theorem~\ref{thm:main} with the
exception of the optimal realizability of
$\dbl J(2)$ and $\dbl J(2)^\vee$:

\begin{thm}[{\cite{baker:joker}}]
The module $\dbl J(k)$ is stably realizable
iff $k \leq 2$ iff $\dbl J(k)^\vee$ is stably
realizable.

For $k \leq 1$, the modules $\dbl J(k)$ and
$\dbl J(k)^\vee$ are optimally realizable.
\end{thm}

The main result of this paper is the case
of $k=2$. We will, however, also give an
alternative proof of the cases where $k<2$,
which may aid as an illustration of how the
more complicated case of $k=2$ works.

For the case of $\dbl J(2)$, we observe that
this module appears as a quotient module of
the rank-$4$ Dickson algebra, which is realized
by the exotic $2$-compact group $B\DW_3$ constructed
by Dwyer and Wilkerson. Our approach to constructing
an optimal realization is to map a suitable skeleton
of $B\DW_3$ to the spectrum $\tmf/2$ of topological
modular forms modulo $2$ so that a skeleton of the
homotopy fiber of this map has cohomology $\dbl J(2)$.
The existence of this map
$\alpha\colon (B\DW_3)^{(24)} \to \underline{\smash{\tmf/2}}_{14}$
is equivalent to the survival of a class
$x_{-15} \in H^{15}((B\DW_3)^{24})$ in the
$\st(2)$-based Adams spectral sequence. This
is the content of Section~\ref{sec:quadrubleJ}.

One might interpret the survival of this class
as (albeit weak) evidence that a faithful
$15$-dimensional spherical homotopy representation
of $\DW_3$ exists, a question we hope to address
in a later paper. The lowest known faithful homotopy
representation of $\DW_3$ at the time of this writing
has complex dimension $2^{46}$ \cite{ziemianski:di-4}.

Our alternative proof for $\dbl J(1)$ in Section~\ref{sec:doubleJ}
follows the same line of reasoning as for $\dbl J(2)$,
but using the rank-$3$ Dickson algebra (realized by
the classifying space of the Lie group $G_2$) and
real $K$-theory instead of topological modular forms.
In this case, we are able to construct the analog of
the map $\alpha$ geometrically.

\subsection*{Conventions}
We assume that all spaces and spectra are completed
at the prime~$2$, although our arguments can be
easily modified to work globally. We will often
assume that we are working with CW complexes which
have been given minimal cell structures. All cohomology
is with coefficients in $\F_2$, the field with two
elements. For spaces, we use unreduced cohomology
but for spectra reduced cohomology; hopefully the
usage should be clear from the context.

\subsection*{Acknowledgements}

The first author would like to thank the following:
The Isaac Newton Institute for Mathematical Sciences
for support and hospitality during the programme
\emph{Homotopy Harnessing Higher Structures} when
work on this paper was undertaken (this work was
supported by EPSRC grant number EP/R014604/1);
Kungliga Tekniska H\"ogskolan and Stockholms
Universitet for supporting a visit by A.~Baker
in Spring of 2018; Okke van Garderen and Sarah
Kelleher for listening. We would also like to
thank the anonymous referee for their comments
that helped us improve this paper.

\section{Realizability of modules over the Steenrod algebra}

Recall that an $\st$-module is called \emph{unstable}
if
\[
\Sq^i(x) = 0 \text{ if } i > |x|.
\]

\begin{defn}
Let $M$ be an $\st$-module. The \emph{unstable degree}
$\sigma(M)$ of $M$ is the minimal $s \in \Z$ such that
$\Sigma^s M$ is an unstable $\st$-module.
\end{defn}

Obviously, $\sigma(M)$ is a finite number 
if $M$ is a nontrivial finite module (but 
may be infinite otherwise). As the anonymous 
referee pointed out, using lower-indexing 
$\Sq_i x = \Sq^{|x|-i}x$ of the Steenrod
squares,
\[
-\sigma(M) = \inf\{i \mid \Sq_i \neq 0\}
\]
since a module is unstable iff $\Sq_i = 0$ 
for all $i<0$.

If a finite module is stably realizable, Freudenthal's theorem implies that
it is realizable by a space after sufficiently 
high suspension (cf. Proposition~\ref{prop:stablyrealizingduals}). 
If $M$ is stably realizable by a spectrum $X$ 
then $M^*$ is stably realized by the Spanier-Whitehead 
dual $DX = F(X,\S^0)$. Only a finite number 
of iterated doubles of $M$ can ever be stably 
realizable by the solution of the Hopf 
invariant~$1$ problem.

\begin{example}
Any $\st$-module $M$ of dimension $1$ over $\F_2$
is optimally realizable (by a point). Let $M$ be
cyclic of dimension $2$ over $\F_2$, thus
$M \cong \F_2\langle \iota,\Sq^{2^i}\iota\rangle$.
By the solution of the Hopf invariant one problem,
$M$ is stably realizable if and only if $i=0,1,2,3$.
In each case, $M$ is optimally realizable by the
projective plane over $\R$, $\CC$, the quaternions,
and the octonions, respectively.
\end{example}

\begin{example}
A simple example of a module that is not optimally 
realizable is the ``question mark'' complex
\[
\begin{tikzpicture}[scale=0.4]
\node (e) at (0,0) [gen]{};
\node (1) at (0,1)[gen]{};
\node (2) at (0,3)[gen]{};
\node at (-1,0) {0};
\node at (-1,1) {1};
\node at (-1,3) {3};
\draw (e) to[sq1] (1);
\draw (1) to[sq2r] (2);
\end{tikzpicture}
\]
or $M=\mathcal A/(\Sq^2,\Sq^3,\dots)$.
This picture, and others to follow, are to 
be read as follows. The numbers on the left 
denote the dimension. A dot denotes a copy 
of $\F_2$ in the corresponding dimension. 
A straight line up from a dot~$x$ to a 
dot~$y$ indicates that $\Sq^1x=y$, and 
a curved line similarly indicates a nontrivial 
operation $\Sq^2$. The unstable degree of 
this module is $1$, but it is not optimally 
realizable because a hypothetical space $X$ 
with $H^*(X)=\Sigma M$, $H^1(X)=\langle x\rangle$ 
would have $x^4=(x^2)^2 \neq 0$ but $x^3=0$.
\end{example}

\section{The family of Jokers}

The finite cyclic $\st$-module
\[
J = \mathcal \st/ (\st\Sq^3+\st\Sq^4\st+\st\Sq^8\st+\cdots)
\]
is called the Joker. Its dimension over $\F_2$ is $5$, 
having dimension $1$ in each degree $0 \leq d \leq 4$; 
a basis is given by
\[
\{1,\Sq^1,\Sq^2,\Sq^2\Sq^1,\Sq^2\Sq^2=\Sq^1\Sq^2\Sq^1\},
\]
or pictorially,
\[
\begin{tikzpicture}[scale=0.6]
\node (e) at (0,0) [gen]{};
\node (1) at (0,1)[gen]{};
\node (2) at (0,2)[gen]{};
\node (21) at (0,3)[gen]{};
\node (22) at (0,4)[gen]{};
\foreach \i in {0,...,4}
{\node at (-1,\i) {\i};
}

\draw (e) to[sq1] (1);
\draw (21) to[sq1] (22);
\draw (e) to[sq2r] (2) to[sq2r] (22);
\draw (1) to[sq2l] (21);
\end{tikzpicture}
\]

The Joker appears in several contexts in 
homotopy theory. In \cite{adams-priddy},
Adams and Priddy showed that $J$ generates 
the torsion on the Picard group of 
$\st(1)$-modules. The Joker also appears 
regularly in projective resolutions of 
cohomologies of common spaces (such as 
real projective spaces) over $\st$ or 
$\st(1)$.
\fxnote{(3): removed comment about Toda bracket}
%Its stable realizability is equivalent to the Toda bracket relation
%$\eta^2 \in \langle 2,\eta,2\rangle$, where $\eta \in \pi_1^s(\S^0)$ denotes
%the Hopf map.
Its linear dual $J^\vee = \Hom(J,\F_2)$ is also a cyclic left
module by the antipode $\chi$ of $\st$. It is not isomorphic
to $J$, even by a shift, since $\chi(\Sq^4)=\Sq^4+\Sq^1\Sq^2\Sq^1$,
which means that $\Sq^4\neq0$ on $J^\vee$. Pictorially,
\[
J^\vee = \begin{matrix}\begin{tikzpicture}[scale=0.6]
\node (e) at (0,0) [gen]{};
\node (1) at (0,1)[gen]{};
\node (2) at (0,2)[gen]{};
\node (21) at (0,3)[gen]{};
\node (22) at (0,4)[gen]{};
\draw (e) to[sq1] (1);
\draw (21) to[sq1] (22);
\draw (e) to[sq2r] (2) to[sq2r] (22);
\draw (1) to[sq2l] (21);
\draw (e) to[sq4l] (22);
\end{tikzpicture}
\end{matrix}
\]
Here and in what follows, the slanted square
brackets denote nontrivial operations $\Sq^4$ or $\Sq^8$.
\fxnote{(8): added sentence}
Note that $J^\vee\cong J$ as $\st(1)$-modules.

The $k$-fold iterated doubles of these
%will be denoted by
are $\dbl J(k)$ and $\dbl J(k)^\vee$, where
$\dbl J(1)$ has
%cells
basis vectors in even dimensions $0$, $2$,
$4$, $8$, $10$, $\dbl J(2)$ has basis
vectors in dimensions divisible by $4$,
and so on.

Clearly, the unstable degrees are given by
$\sigma(\dbl J(k))=2 \cdot 2^k$ (the bottom
cohomology class supports a nontrivial
operation $\Sq^{2\cdot 2^k}$), and
$\sigma(\dbl J(k)^\vee)= 4 \cdot 2^k$ (the
bottom cohomology class supports also a
nontrivial operation $\Sq^{4 \cdot 2^k}$).

Note that if $\dbl J(k)$ is optimally realized
by a space, then that space is weakly equivalent
to a CW complex $X$ with cells in dimensions
$i\cdot 2^k$, where $i=2,\dots,6$, hence $X$
has dimension $6\cdot 2^k$. The ring structure
of the cohomology is implied by the instability
condition for $\st$-algebras, namely,
$\Sq^i(x)=x^2$ when $|x|=i$:
\[
H^*(X) = \F_2[x_2,x_3]/(x_2,x_3)^3
\quad (|x_i|=i\cdot 2^k).
\]
If $\dbl J(k)^\vee$ is optimally realizable by
a space, then that space is weakly equivalent
to a CW complex $Y$ with cells in dimensions
$j \cdot 2^k$, where $j=4,5,6,7,8$. For
dimensional reasons, the ring structure of
the cohomology has to be
\[
H^*(Y) =
\F_2[x_4,x_5,x_6,x_7]/(x_4^3)+x_4(x_5,x_6,x_7)+(x_5,x_6,x_7)^2;
\quad (|x_i|=i\cdot 2^k).
\]

\section{Dual Jokers}

If $X$ is a spectrum with $H^*(X) \cong \dbl J(k)$
then it is obvious that the Spanier-Whitehead
dual $DX$ realizes $\dbl J(k)^\vee$, up to
a degree shift, i.e.,
\[
\Sigma^{4\cdot 2^k} H^*(DX) \cong \dbl J(k)^\vee.
\]

Unstably, the situation is a bit more complicated, but follows from a more general consideration.

\begin{lemma}\label{lemma:topdegreecorrection}
Let $M$ be a finite $\st$-module with top nonvanishing degree~$n$ and $Y$ a space with an injective $\st$-module map $f\colon M \to H^*(Y)$ whose cokernel is $n-1$-connected. Then there is a space $Z$ such that $H^*(Z) \cong M$ as $\st$-modules.
\end{lemma}
\begin{proof}
Let $V$ be a complement of $\im(f)$ in $H^n(Y)$ and denote by $\alpha\colon Y \to K(V,n)$ its representing map. Let $Z$ be the $n$-skeleton of the homotopy fiber of $\alpha$. Then $H^*(Z) \cong M$.
\end{proof}

\begin{prop}\label{prop:stablyrealizingduals}
Let $M$ be a finite, stably realizable, nonnegatively 
graded $\st$-module with top nonvanishing degree~$n$. 
Then $\Sigma^n M = H^*(Z)$ for some CW complex~$Z$.
\end{prop}
\begin{proof}
Let $X$ be a spectrum such that $M = H^*(X)$ and 
consider the space $Y = \Omega^\infty \Sigma^n X$.

Since for any $k$-connected spectrum $E$, the 
augmentation $\Sigma^\infty\Omega^\infty E \to E$ 
is $(2k+2)$-connected, the map 
$\Sigma^\infty Y \to \Sigma^n X$ is 
$2(n-1)+2 = 2n$-connected. Hence the induced 
map $H^i(\Sigma^n X) \to H^i(Y)$ is an isomorphism 
for $i<2n$ and injective for $i=2n$. By 
Lemma~\ref{lemma:topdegreecorrection}, there 
exists a space~$Z$ such that 
$H^*(Z)\cong H^*(\Sigma^n X)$ as $\st$-modules.
\end{proof}

\begin{corollary}
If $\dbl J(k)$ is stably realizable for any~$k$ 
then $\dbl J(k)^\vee$ is optimally realizable.
\end{corollary}
\begin{proof}
The module $\dbl J(k)$ has top nonvanishing degree 
$4\cdot 2^k$, so $M=\Sigma^{4\cdot 2^k} \dbl J(k)^\vee$ 
satisfies the condition of Proposition~\ref{prop:stablyrealizingduals} 
for $n=4\cdot 2^k$. Hence there is a space $Z$ such 
that $H^*(Z) \cong \Sigma^{8\cdot 2^k} \dbl J(k)^\vee$. 
Then $Z$ has its bottom cell in degree 
\[
4 \cdot 2^k = \sigma(\dbl J(k)^\vee) = 4 \cdot 2^k,
\] 
proving the claim.
\end{proof}

Applying Proposition~\ref{prop:stablyrealizingduals} 
to a stably realized $\dbl J(k)$ gives a space $Z$ 
such that $H^*(Z) \cong \Sigma^{4 \cdot 2^k} \dbl J(k)$, 
but since $\sigma(\dbl J(k)) = 2 \cdot 2^k$, this 
does not suffice to prove optimal realizability 
of $\dbl J(k)$. This is why the following sections 
are needed.

\section{Dickson algebras and their realizations}

The rank-$n$ algebra of Dickson invariants $DI(n)$
is the ring of invariants of
$\Sym(\F_2^n) = \F_2[t_1,\dots,t_n]$ under the
action of the general linear group $\GL_n(\F_2)$.
We think of $\Sym(\F_2^n)$ as a graded commutative
ring with $t_i$ in degree $1$.
Dickson~\cite{dickson:fundamental} showed that
\[
DI(n) \cong \F_2[x_{2^n-2^i} \mid 0 \leq i < n],
\]
where subscripts denote degrees. The polynomials $x_{2^n-2^i}$ are
given by the formula
\[
\prod_{v \in \F_2^n} (X + v)
= \sum_{i=0}^n x_{2^n-2^i} X^{2^i} \in \Sym(\F_2^n)[X],
\]
where $x_0 = 1$ by convention. If we give $\Sym(\F_2^n)$ the structure of an
$\st$-algebra with $\Sq(t_i) = t_i+t_i^2$
(i.e., by using the isomorphism
$\Sym(\F_2^n) \cong H^*(B\F_2^n)$) then $DI(n)$
is an $\st$-subalgebra with
\[
\Sq^{2^i}x_{2^n-2^{i+1}}=x_{2^n-2^i}.
\]

\begin{thm}[Smith-Switzer, Lin-Williams, Dwyer-Wilkerson]
The Dickson algebra $DI(n)$ is optimally realizable iff
$n \leq 4$. \qed
\end{thm}

The first three Dickson algebras are realized by $\R P^\infty$, $BSO(3)$, and $BG_2$ (the classifying space of the exceptional Lie group $G_2$), respectively. The case $n=4$ was
settled in \cite{dwyer-wilkerson:dw3}, where Dwyer
and Wilkerson constructed a $2$-complete space, the
exceptional $2$-compact group $B\DW_3$, with the
required cohomology.

A graphical representation of a skeleton of the
spaces realizing the Dickson algebras is given
below. One observes that the Jokers $\dbl J(i)$
occur as quotients of skeleta of these spaces;
the kernel consists of the classes on the right
of each diagram. However, realizing these
quotients as fibers of certain maps is non-obvious
and the purpose of the following section.

\begin{minipage}{.5\textwidth}
\begin{align} \label{eq:joker1}
BSO(3)\colon & \begin{matrix}
\begin{tikzpicture}[scale=0.6]
\node (x2) at (0,0) [gen,label=180:$x_2$]{};
\node (x3) at (0,1)[gen,label=180:$x_3$]{};
\node (x2^2) at (0,2)[gen,label=0:$x_2^2$]{};
\node (x2x3) at (0,3)[gen,label=180:$x_2x_3$]{};
\node (x3^2) at (0,4)[gen,label=180:$x_3^2$]{};
\node[color=red] (x2^3) at (1,4)[gen,label=0:$x_2^3$]{};
\draw (x2) to[sq1] (x3);
\draw (x2x3) to[sq1] (x3^2);
\draw (x2) to[sq2r] (x2^2) to[sq2r] (x3^2);
\draw (x3) to[sq2l] (x2x3);
\end{tikzpicture}
\end{matrix}\\
\label{eq:joker2}
B\GG_2\colon & \begin{matrix}
\begin{tikzpicture}[scale=0.6]
\node (x4) at (0,0) [gen,label=180:$x_4$]{};
\node (x6) at (0,2)[gen,label=180:$x_6$]{};
\node[color=orange] (x7) at (1,3)[gen,label=0:$x_7$]{};
\node (x4^2) at (0,4)[gen,label=0:$x_4^2$]{};
\node (x4x6) at (0,6)[gen,label=180:$x_4x_6$]{};
\node[color=orange] (x4x7) at (1,7)[gen,label=0:$x_4x_7$]{};
\node (x6^2) at (0,8)[gen,label=180:$x_6^2$]{};
\node[color=red] (x4^3) at (1,8)[gen,label=0:$x_4^3$]{};
\draw (x4) to[sq2l] (x6);
\draw[color=orange] (x6) to[sq1] (x7) to[sq4r] (x4x7);
\draw (x4x6) to[sq2l] (x6^2);
\draw[color=orange] (x4x6) to[sq1] (x4x7);
\draw (x4) to[sq4r] (x4^2) to[sq4r] (x6^2);
\draw (x6) to[sq4l] (x4x6);
\end{tikzpicture}
\end{matrix}
\end{align}
\end{minipage}%
\begin{minipage}{.5\textwidth}
\begin{equation} \label{eq:joker3}
B\DW_3\colon\begin{matrix}
\begin{tikzpicture}[scale=0.6]
\node (x8) at (0,0) [gen,label=180:$x_8$]{};
\node (x12) at (0,4)[gen,label=180:$x_{12}$]{};
\node[color=orange] (x14) at (1.8,6)[gen,label=0:$x_{14}$]{};
\node[color=orange] (x15) at (1.8,7)[gen,label=0:$x_{15}$]{};
\node (x8^2) at (0,8)[gen,label=0:$x_8^2$]{};
\node (x8x12) at (0,12)[gen,label=180:$x_8x_{12}$]{};
\node[color=orange] (x8x14) at (1.8,14)[gen,label=0:$x_8x_{14}$]{};
\node[color=orange] (x8x15) at (1.8,15)[gen,label=0:$x_8x_{15}$]{};
\node (x12^2) at (0,16)[gen,label=180:$x_{12}^2$]{};
\node[color=red] (x8^3) at (1.8,16)[gen,label=0:$x_8^3$]{};

\draw (x8) to[sq4l] (x12);
\draw[color=orange] (x12) to[sq2r] (x14) to[sq8r] (x8x14) to[sq1] (x8x15);
\draw[color=orange] (x14) to[sq1] (x15) to [sq8l] (x8x15);
\draw (x8x12) to[sq4l] (x12^2);
\draw[color=orange] (x8x12) to[sq2r] (x8x14);
\draw (x8) to[sq8r] (x8^2) to[sq8r] (x12^2);
\draw (x12) to[sq8l] (x8x12);
\end{tikzpicture}
\end{matrix}
\end{equation}
\end{minipage}

%\section{The Joker \texorpdfstring{$J$}{J}
%and its double \texorpdfstring{$\dbl J(1)$}{J(1)}}
\section{The Joker $J$ and its double} \label{sec:doubleJ}

The cohomology picture \eqref{eq:joker1} shows
that the $6$-skeleton of $B\SO(3)$ is almost
a realization of $J=\dbl J(0)$, its only defect
lying in an additional class $x_2^3$ in the top
cohomology group $H^6(B\SO(3))$. Let
$\alpha\colon B\SO(3) \to K(\F_2,6)$ represent
this class and $X = \hofib(\alpha)^{(6)}$, the
$6$-skeleton of its homotopy fiber. Then $X$
realizes $\dbl J(0)$ optimally.

For the double Joker $\dbl J(1)$, as seen in
the cohomology picture \eqref{eq:joker2}, it
does not suffice any longer to take a skeleton
of $B\GG_2$ and kill off a top-dimensional 
class. 
\fxnote{(13) Added some justification/explanation}
Since we feel that the ideas that come up here
led us to the work appearing in Section~\ref{sec:quadrubleJ}
we feel it worth describing them in some detail.

First we recall some standard results on the 
exceptional Lie group~$\GG_2$ and its relationship 
with $\Spin(7)$. 

One definition of~$\GG_2$ is as the group of 
automorphisms of the alternative division ring 
of Cayley numbers (octonions)~$\O$. Since~$\GG_2$ 
fixes the real Cayley numbers, it is a closed 
subgroup of $\SO(7)\leq\SO(8)$.

A different point of view is to consider the 
spinor representation of $\Spin(7)$. Recall 
that the Clifford algebra $\Cl_6\cong\Mat_8(\R)$
is isomorphic to the even subalgebra of 
$\Cl_7\cong\Mat_8(\R)\times\Mat_8(\R)$, so 
$\Spin(7)$ is naturally identified with a 
subgroup of $\SO(8)\subseteq\Mat_8(\R)$, 
and thus acts on~$\R^8$ with its spinor
representation. Then on identifying $\R^8$ 
with~$\O$, we find that the stabilizer 
subgroup in $\Spin(7)$ of a non-zero vector 
is isomorphic to~$\GG_2$. It follows that 
the natural fibration
\[
\Spin(7)/\GG_2\to B\GG_2\to B\Spin(7)
\]
is the unit sphere bundle of the associated 
spinor vector bundle $\sigma\to B\Spin(7)$.

The mod-$2$ cohomologies of these spaces 
are related as follows. By considering 
the natural fibration
\[
K(\F_2,1)\to B\Spin(7)\to B\SO(7)
\]
we find that
\[
H^*(B\Spin(7)) = \F_2[w_4,w_6,w_7,u_8]
\]
where the $w_i$ are the images of the 
universal Stiefel-Whitney classes in
\[
H^*(B\SO(7)) = \F_2[w_2,w_3,w_4,w_5,w_6,w_7],
\]
and $u_8\in H^8(B\Spin(7))$ is detected 
by $z_1^8\in H^8(K(\F_2,1))$. It is known 
that
\[
H^*(B\GG_2) = \F_2[x_4,x_6,x_7]
\]
and it is easy to see that the generators 
can be taken to be the images of $w_4,w_6,w_7$
under the induced homomorphism $H^*(B\Spin(7))\to H^*(B\GG_2)$.
As a consequence, these $x_i$ are Stiefel-Whitney
classes of the pullback $\rho_7\to B\GG_2$
of the natural $7$-dimensional bundle
$\rho\to B\SO(7)$ and since this lifts to
a $\Spin$ bundle, it admits an orientation
in real connective $K$-theory. This leads
to the following observation.

\begin{lemma}\label{lemma:BG2factorization}
There is a factorisation
\[
B\GG_2 \to \underline{\kO}_7 \to K(\F_2,7)
\]
of a map representing $x_7\in H^7(B\GG_2)$.
\end{lemma}

Here
\[
\underline{\kO}_7=\Omega^\infty\Sigma^7\kO\sim\Omega B\mathrm{O}\langle8\rangle.
\]
and $\underline{\kO}_7\to K(\F_2,7)$ is
the infinite loop map induced from the
unit morphism $\kO\to H\F_2$.

The cohomology of $B\mathrm{O}\langle8\rangle$ is a quotient
of that of~$B\mathrm{O}$:
\begin{multline}\label{eq:BO8}
H^*(B\mathrm{O}\langle8\rangle) =
\F_2[w_{2^r}:r\geq3]\otimes\F_2[w_{2^r+2^{r+s}}:r\geq2,\,s\geq1] \\
\qquad
\otimes\F_2[w_{2^r+2^{r+s}+2^{r+s+t}}:r\geq1,\,s,t\geq1] \\
\otimes\F_2[w_{2^r+2^{r+s}+2^{r+s+t}}:r\geq0,\,s,t\geq1],
\end{multline}
where the $w_i$ are images of universal Stiefel-Whitney
classes in~$H^*(B\mathrm{O})$. Here
\[
\Sq^4w_8\equiv w_{12}\pmod{\text{decomposables}}.
\]

A routine calculation shows that 
$H^*(\underline{\kO}_7) \cong H^*(\Omega B\mathrm{O}\langle8\rangle)$
is the exterior algebra on certain elements~$e_i\in H^i(\underline{\kO}_7)$
where $e_i$ suspends to the generator $w_{i+1}$ of~\eqref{eq:BO8}.

In particular, up to degree~$13$,
\[
H^*(\underline{\kO}_7)=\F_2\{1,e_7,e_{11}\}
\]
and
\begin{equation}\label{eq:Sq4e8}
\Sq^4e_7 = e_{11}.
\end{equation}

\begin{lemma}\label{lem:BG2fibre}
The module $\dbl J(1)$ is optimally realizable.
\end{lemma}

\begin{proof}
Let $\alpha\colon B\GG_2 \to \underline{\kO}_7$
be the factorization of Lemma~\ref{lemma:BG2factorization}.
By the above computations, $H^*(B\GG_2)$, as an
algebra over $H^*(\underline{\kO}_7)$, is
isomorphic to
\[
H^*(B\GG_2) \cong
H^*(\underline{\kO}_7)
[x_4,x_6,\text{generators in degree greater than $12$}]/R
\]
where the module $R$ of relations is at least
$13$-connected. This means that in the
Eilenberg-Moore spectral sequence for the
cohomology of the fiber of $\alpha$, $E_2^{s,t}=0$
for $s+t\leq 12$ and $s<0$. Thus up to degree $12$,
\[
H^*(\hofib(\alpha)) \cong \F_2[x_4,x_6].
\]
An application of Lemma~\ref{lemma:topdegreecorrection} takes care of the remaining top class $x_4^3$ and shows that $\dbl J(1)$ is optimally realizable.
\end{proof}

%\section{The quadruple Joker \texorpdfstring{$J(2)$}{J(2)}}
\section{The quadruple Joker} \label{sec:quadrubleJ}

The strategy to construct an optimal realization
of $\dbl J(2)$ consists of an easy and a harder step. The easy step is to construct a space $Y$ whose cohomology is diagram~\ref{eq:joker3} without the topmost unattached class:

\begin{lemma} \label{lemma:Yspace}
There exists a space $Y$ with
\[
H^*(Y) = \F_2[x_8,x_{12},x_{14},x_{15}] / (x_8^3,\text{polynomials of degree $>24$})
\]
\end{lemma}
\begin{proof}
This follows from an application of Lemma~\ref{lemma:topdegreecorrection} to the $24$-skeleton of $B\DW_3$.
\end{proof}

The harder step is to realize $\dbl J(2)$ as a skeleton of the homotopy fiber of a suitable map
\[
\alpha\colon Y \to \underline{\smash{\tmf/2}}_{14}
\]
into the $14$th space of the spectrum of topological
modular forms modulo~$2$. The spectrum $\tmf$ is an
analog of connective real $K$-theory, $\kO$, but
of chromatic level~$2$ \cite{hopkins-mahowald,tmfbook,behrens:notes-on-tmf,goerss:tmfsurvey}
with well-known homotopy \cite{bauer:tmf}.

\begin{prop}\label{prop:existenceofbeta}
Let $Y$ be a space as in Lemma~\ref{lemma:Yspace}. Then there exists a $2$-torsion class
\[
\beta \in \tmf^{15}(Y)
\]
whose classifying map induces an isomorphism of the order-$2$ groups
\[
H^{15}(\underline{\tmf}_{15}) \to H^{15}(Y).
\]
\end{prop}

\begin{proof}[Proof of Thm.~\ref{thm:main}]
Given Prop.~\ref{prop:existenceofbeta} and the Bockstein spectral sequence, the class $\beta$ has to pull back to a class $\alpha \in (\tmf/2)^{14}(Y)$ whose classifying map induces an isomorphism in $H^{14}$. This means that under
$\alpha^*\colon H^*(\underline{\smash{\tmf/2}}_{14}) \to H^*(Y)$, the unit $\iota \in H^{14}(\underline{\smash{\tmf/2}}_{14})$ is mapped to $x_{14}$.

A basic property of $\tmf$ is that
$H^*(\tmf) \cong \st \otimes_{\st(2)} \F_2$
and so there is a non-split extension of
$\mathcal{A}$-modules
\[
0\to H^*(\tmf)\to H^*(\tmf/2)\to \Sigma H^*(\tmf) \to 0
\]
where $Sq^1$ acts non-trivially on the generator of
$\Sigma H^0(\tmf)$.

This implies that
$\alpha^*(\Sq^1 \iota) = x_{15}$,
$\alpha^*(\Sq^8\iota) = x_8x_{14}$, and
$\alpha^*(\Sq^8\Sq^1\iota) = x_8x_{15}$.
Hence as in the case of $\dbl J(1)$,
\[
H^*(Y) \cong
H^*(\underline{\smash{\tmf/2}}_{14})
[x_8,x_{12},\text{generators in degree greater than $24$}]/(x_8^3,R),
\]
where the module $R$ of relations is at least
$25$-connected. Thus the Eilenberg-Moore spectral
sequence converging to $\hofib(\alpha)$ shows that
up to degree $24$,
\[
H^*(\hofib(\alpha)) \cong \F_2[x_8,x_{12}]/(x_8^3),
\]
The $24$-skeleton of $\hofib(\alpha)$ therefore optimally realizes $\dbl J(2)$.
\end{proof}

It remains to prove Prop.~\ref{prop:existenceofbeta}.

Let $Y_m^n = Y^{(n)}/Y^{(m-1)}$ denote the $n$-skeleton of $Y$ modulo
the $(m-1)$-skeleton, thus containing the cells from dimension $m$ to
dimension~$n$.

\begin{lemma}\label{lemma:basicd2}
Let $Y$ be a space as in Lemma~\ref{lemma:Yspace}. Then the space $Y_{16}^{20}$ is homotopy equivalent to a suspension of the
cone of $\pm 2\nu$. In particular, in the Adams spectral sequence
\[
\Ext_{\st(2)}(\F_2,H^{-*}(Y_{16}^{20}) \Longrightarrow \tmf^{-*}(Y_{16}^{20}),
\]
there is a differential $d^2(x_{-16}) = x_{-20} h_0 h_2$, where $x_{-16}$,
$x_{-20}$ in $\Ext^0$ are the classes corresponding to the two cells.
\end{lemma}
Here the grading is chosen such that the spectral sequence becomes
a homological spectral sequence and we will display it in the Adams
grading.
\begin{proof}
Consider the space $Y^{20}$, the $20$-skeleton of $Y$. In the Atiyah-Hirzebruch spectral sequence
\[
H^{-*}(Y^{20},\tmf^{-*}) \cong H_*(D(Y^{20}),\tmf_*) \Longrightarrow \tmf^{-*}(Y^{20}),
\]
the cohomology generators $x_8$, $x_{12}$ represent classes $x_{-8}, x_{-12} \in H_*(DY^{20},\tmf_0)$ and, since $\Sq^4(x_8)=x_{12}$, there is a differential $d^4(x_{-8}) = x_{-12}\nu \pmod {2\nu}$. By multiplicativity, $d^4(x_{-8}^2) = 2x_{-8}x_{-12} \nu \pmod {4\nu}$. This shows that the top cell of $Y^{20}$ is attached to the $16$-dimensional cell by $2\nu \pm 4{\nu} = \pm 2 \nu$.
\end{proof}

\begin{proof}[Proof of Prop.~\ref{prop:existenceofbeta}]
The claim boils
down to showing that in the Adams spectral sequence
\[
E_2^{s,t} = \Ext_{\st}(H^{-*}(\tmf),H^{-*}(Y))
\cong \Ext_{\st(2)}(\F_2,H^{-*}(Y))
\Longrightarrow \tmf^{-*}(Y),
\]
the unique nontrivial class
\[
x_{-15} \in E_2^{0,-15}
= \Hom_{\st}(H^0(\tmf),H^{15}(Y))
\]
is an infinite cycle.

As modules over $\st(2)$, 
\[
H^*(Y) \cong H^*(Y_0^{15}) \oplus H^*(Y_{16}^{24}),
\]
hence the $E_2$-term above splits as a sum as well.
The following is the $E_2$-term of the Adams spectral 
sequence converging to  $\tmf^{-*}(Y^{(15)})$, 
determined with Bob Bruner's program~\cite{bruner:extsoftware}:
\[
\includegraphics[clip,trim=4.5cm 19.5cm 8cm 4cm]{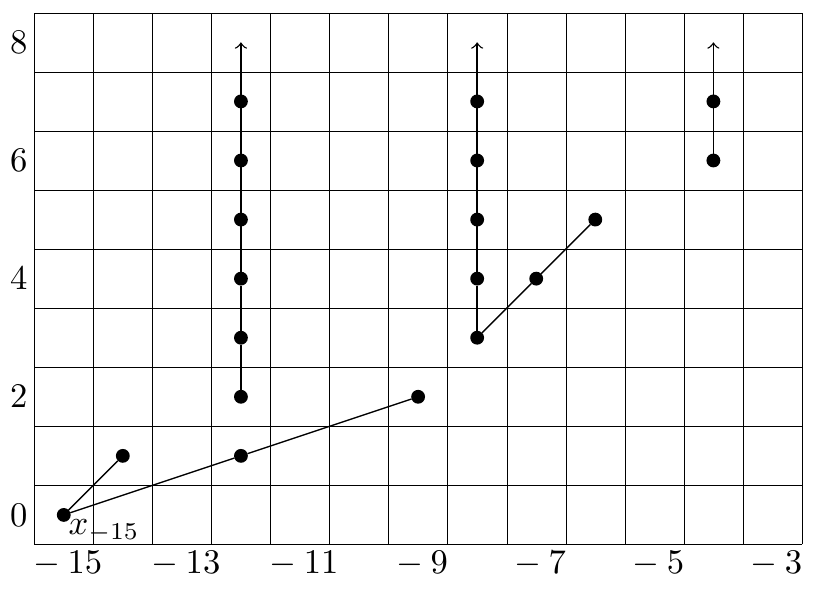}
\]
Similarly, the $E_2$-term of the spectral sequence 
computing $\tmf^*(Y_{16}^{24})$ is the following.
\[
\includegraphics[clip,trim=4.5cm 19.5cm 8cm 4cm]{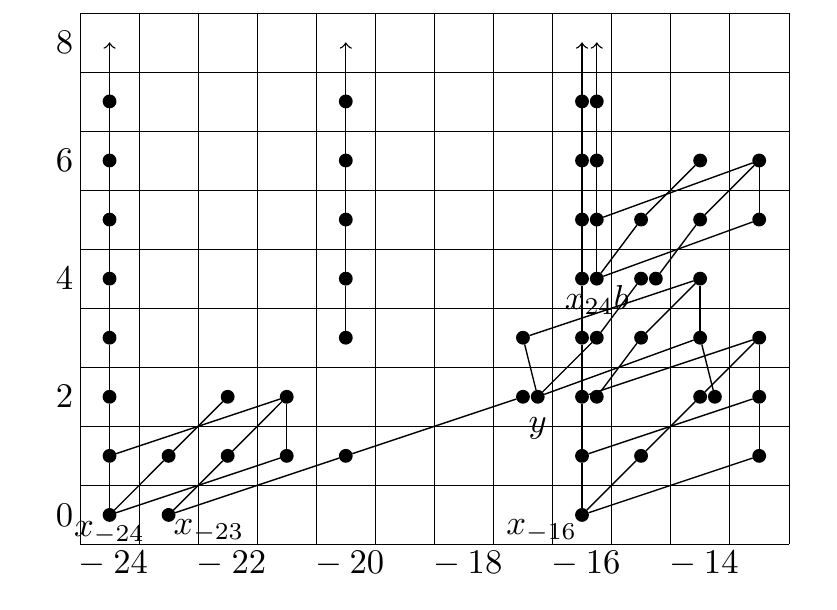}
\]
We will identify the possible targets of differentials on
$\iota$. Since $h_0\iota = 0$ and $h_1^2\iota=0$, only
classes in the kernel of $h_0$ and the kernel of $h_1^2$
can be targets. This means that there is no possible target
for a $d_2$ in bidegree $(-16,2)$ in the displayed Adams
grading.

It is easy to see that under the inclusion
$Y_{16}^{20} \to Y_{16}^{24}$, the class $x_{-20}h_0h_2$
maps to the displayed class $y$. By Lemma~\ref{lemma:basicd2},
there is thus a differential $d_2(x_{-16}) = y$. This
implies that the $E_3$-term for $Y_{16}^{24}$ is given
by the following chart.
\[
\includegraphics[clip,trim=4.5cm 19.5cm 8cm 4cm]{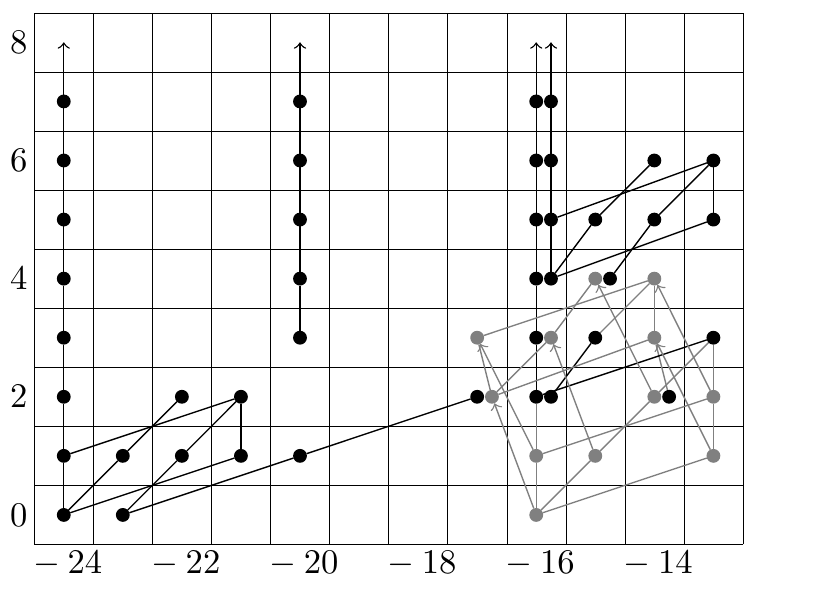}
\]
The only remaining possible target of a $d_3$ in bidegree
$(-16,3)$ is $x_{-16}h_0^3$, which is impossible for
the same reason as before ($h_0$ on it is nontrivial.)

There are no longer differentials possible on $\iota$
either because no classes in filtration $4$ or higher
are $h_0$-torsion in any $E_n$-term.
\end{proof}

\bibliographystyle{amsalpha}
\bibliography{bibliography}
\end{document}